\documentclass{amsart} 
\title[Milnor operations]{Milnor operations and classifying spaces}
\author{Masaki Kameko}
\address{Department of Mathematical Sciences,
College of Systems Engineering and Science, 
Shibaura Institute of Technology,
307 Minuma-ku Fukasaku, Saitama-City 337-8570, Japan}
\email{kameko@shibaura-it.ac.jp}
\thanks{This work was supported by JSPS KAKENHI Grant Number JP17K05263.}
\usepackage{amssymb,amscd, pb-diagram}
\usepackage[initials, alphabetic, nobysame]{amsrefs}
\newtheorem{theorem}{Theorem}[section]
\newtheorem{proposition}[theorem]{Proposition}
\newtheorem{lemma}[theorem]{Lemma}

\newtheorem{conjecture}[theorem]{Conjecture}
\theoremstyle{definition}

\newtheorem{remark}[theorem]{Remark}
\newtheorem{example}[theorem]{Example}

\begin{document}
\maketitle

\begin{abstract}
We give an example of a nonzero odd degree element of the classifying space of a connected Lie group 
such that all higher Milnor operations vanish on it. It is a counterexample of a conjecture of Kono and Yagita.
\end{abstract}


\section{Introduction}   \label{section:1}

For each prime number $p$, there are the mod $p$ and Brown-Peterson cohomology. 
For a compact connected Lie group $G$,  the mod $p$ cohomology of the classifying space $BG$ has no nonzero odd degree element if the integral cohomology of  $G$ has no $p$-torsion.  So does the Brown-Peterson cohomology. On the one hand, if the integral homology of $G$ has $p$-torsion, the mod $p$ cohomology of $BG$ has a nonzero odd degree element. On the other hand,  for the Brown-Peterson cohomology, Kono and Yagita conjectured
the following:


\begin{conjecture}[Kono and Yagita, (1) in Conjecture 4 in \cite{kono-yagita-1993}]
\label{conjecture:1.1}
There is no nonzero odd degree element in the Brown-Peterson cohomology of the classifying space of a compact Lie group.
\end{conjecture}

Conjecture~\ref{conjecture:1.1}  is interesting in conjunction with Totaro's conjecture on the cycle map from the Chow ring of the classifying space of a complex linear algebraic group $G$ to its Brown-Peterson cohomology. In \cite{totaro-1997}, Totaro showed that the cycle map from the Chow ring of a complex smooth algebraic variety to its ordinary cohomology factors through the Brown-Peterson cohomology after localized at $p$.  In \cite{totaro-1999}, he defined the Chow ring $CH^*(BG)$ of a linear algebraic group $G$ and conjectured the following.


\begin{conjecture}[Totaro, p.250 in \cite{totaro-1999}]\label{conjecture:1.2}
For a complex linear algebraic group $G$,
if there is no nonzero odd degree element in the Brown-Peterson cohomology $BP^{*}(BG)$, the cycle map
\[
CH^i(BG)_{(p)} \to (\mathbb{Z}_{(p)} \otimes_{BP^{*}} BP^{*}(BG))^{2i}
\]
 is an isomorphism.
\end{conjecture}

With Conjectures~\ref{conjecture:1.1} and \ref{conjecture:1.2}, we expect a close connection between the Chow ring in algebraic geometry and the Brown-Peterson cohomology in algebraic topology. In \cite{kono-yagita-1993}, Kono and Yagita confirmed Conjecture~\ref{conjecture:1.1} for some compact connected Lie groups with $p$-torsion by computing the Atiyah-Hirzebruch spectral sequences. However, the non-triviality of Milnor operations on odd degree elements yields non-trivial differentials sending odd degree elements to non-zero elements, so odd degree elements do not survive to the $E_\infty$-term. With their computational results on the Brown-Peterson cohomology of classifying spaces, Kono and Yagita conjectured the following:


\begin{conjecture}[Kono and Yagita, {Conjecture 5} in \cite{kono-yagita-1993}] \label{conjecture:1.3}
For each nonzero odd degree element $x$ of the mod $p$ cohomology of the classifying space of a compact connected Lie group, there exists an integer $i$ such that for $m\geq i$, 
\[
Q_m x\not=0.
\]
\end{conjecture}


Conjecture~\ref{conjecture:1.3} is interesting in the cohomology theory of classifying spaces of non-simply connected Lie groups. In \cite{vavpetic-viruel-2005}, Vavpeti\v{c} and Viruel showed that if $p$ is an odd prime,  Conjecture~\ref{conjecture:1.3} holds for the projective unitary group $PU(p)$. Moreover, the cohomology of classifying spaces of non-simply connected Lie groups has recently enjoyed renewed interest. Many mathematicians have studied it in various contexts. Antieau, Gu and Williams  (\cite{antieau-williams-2014}, \cite{gu-2019}, \cite{gu-2020}, \cite{gu-2022}) studied it for the topological period-index problem. Antieau, the author and Tripathy (\cite{antieau-2016}, \cite{kameko-2015}, \cite{kameko-2017}, \cite{tripathy-preprint}) studied it for integral Hodge conjecture modulo torsion. Furthermore, the Atiyah-Hirzebruch spectral sequence is used in theoretical physics to study anomalies, cf. Garc\'{\i}a-Etxebarria and Montero \cite{GM-2019}.


In this paper, we give a counterexample for Conjecture~\ref{conjecture:1.3} in the case $p=2$. 
Our result is as follows: Let $\mathbb{H}$ be the quaternions. Let $Sp(1)\subset \mathbb{H}$ be the symplectic group consisting of unit quaternions.
Let $G$ be the quotient of the 3-fold product $Sp(1)^3$ of the symplectic groups $Sp(1)$ 
by the subgroup $\Gamma_2$ generated by $(-1, -1, 1)$ and $(-1, 1 , -1)$.

\begin{theorem}\label{theorem:1.4}
In the mod $2$ cohomology of the classifying space of the compact connected Lie group $G$ above, there exists a nonzero element 
$x_{13}$ of degree $13$ such that 
\[
Q_m x_{13}=0
\]
for $m \geq 1$.
\end{theorem}

This paper is organized as follows. In Section~\ref{section:2}, we describe the action of Milnor operations on the mod $2$ cohomology of $BSO(3)$.
In Section~\ref{section:3}, we prove Theorem~\ref{theorem:1.4} as Proposition~\ref{proposition:3.5}.

The author would like to thank the anonymous referee for suggestions to improve the exposition of this paper. 

\section{Milnor operations}   \label{section:2}


In this section, we recall Milnor operations \[
Q_m\colon H^{i}(X;\mathbb{Z}/2)\to H^{i+2^m+1}(X;\mathbb{Z}/2)
\]
and the mod $2$ cohomology of the classifying space $BSO(3)$. 
Milnor operations $Q_m$ are defined by
\[
Q_0=\mathrm{Sq}^1, \quad Q_m=\mathrm{Sq}^{2^{m}} Q_{m-1}+Q_{m-1} \mathrm{Sq}^{2^m} \quad (m\geq 1).
\]
They have the following properties:
\begin{align*}
Q_m Q_n &= Q_n Q_m, \\
 Q_m^2 &=0, 
\end{align*}
and 
\begin{align*}
Q_m(x \cdot y)&=(Q_m x) \cdot y+ x \cdot (Q_m y).
\end{align*}
These formulae are essential in our proofs Propositions~\ref{proposition:2.2} and \ref{proposition:3.5}.
The mod $2$ cohomology of $BSO(3)$ is a polynomial ring
\[
H^{*}(BSO(3);\mathbb{Z}/2)=\mathbb{Z}/2[w_2, w_3]
\]
generated by two elements $w_2$, $w_3$ of degree $2$, $3$, respectively.
The action of Steenrod squares on these elements is well-known as the Wu formula. In particular, we have
\begin{align*}
\mathrm{Sq}^1 w_2&=w_3, & \mathrm{Sq}^2 w_2 &=w_2^2, \\
\mathrm{Sq}^1 w_3&=0, & \mathrm{Sq}^2 w_3 &=w_2 w_3.
\end{align*}
By the Wu formula and by the definition and elementary properties of Milnor operations stated above, 
 it is easy to obtain 
\begin{align*}
Q_0 w_2&= w_3, & Q_1 w_2&=w_2 w_3, & 
Q_0 Q_1 w_2&=w_3^2,\\
Q_0 w_3&= 0, & Q_1 w_3&=w_3^2, & 
Q_0 Q_1 w_3&=0.
\end{align*}
This section aims to prove the following lemma on the action of Milnor operations on the mod $2$ cohomology of $BSO(3)$. 


\begin{lemma}\label{lemma:2.1}
For $m\geq 2$, there exists a polynomial $g_m$ in $w_2^2$ and $w_3^2$ such that 
we have 
\[
Q_m Q_1w_2=g_m w_3^4.
\]
in the mod $2$ cohomology of $BSO(3)$.
\end{lemma}


To prove Lemma~\ref{lemma:2.1}, we recall the relation between Dickson invariants and Milnor operations as Proposition~\ref{proposition:2.2}. 
The connection between Dickson invariants and Milnor operations 
is an exciting subject in algebraic topology. Thus, we refer the reader to the classical work of Adams and Wilkerson (\cite{adams-wilkerson-1980}, \cite{wilkerson-1983}) for more detail on the background of this section. However, to make this paper self-contained as far as possible, we give detailed proof for Lemma~\ref {lemma:2.1} without mentioning Dickson invariants and the above background. 

Let $(\mathbb{Z}/2)^2=\mathbb{Z}/2\times \mathbb{Z}/2$ be the elementary abelian $2$-subgroup of $SO(3)$ generated by
diagonal matrices
\[
\begin{pmatrix} -1 & 0 & 0 \\ 0 & -1 & 0 \\ 0 & 0 & 1 \end{pmatrix}
, 
\quad
\begin{pmatrix} -1 & 0 & 0 \\ 0 & 1 & 0 \\ 0 & 0 & -1 \end{pmatrix}.
\]
We denote by $\iota\colon (\mathbb{Z}/2)^2 \to SO(3)$ the inclusion map.
The induced homomorphism 
\[
B\iota^{*}\colon H^{*}(BSO(3);\mathbb{Z}/2)\to H^{*}(B(\mathbb{Z}/2)^2;\mathbb{Z}/2)
\]
is injective, and its image is the subring generated by the following elements.
\begin{align*}
B\iota^{*}(w_2)&=s_1^2 + s_1 s_2 + s_2^2,
\\
B\iota^{*}(w_3)&=s_1^2s_2 + s_1 s_2^2.
\end{align*}


\begin{proposition}\label{proposition:2.2}
Suppose that $m\geq 2$.
For an element $x$ of the mod $2$ cohomology of $B(\mathbb{Z}/2)^2$, 
let
\[
D_m x=Q_m x + B\iota^{*}(w_2^{2^{m-1}}) Q_{m-1}x +B\iota^{*}(w_3^{2^{m-1}}) Q_{m-2}x.
\]
Then, we have
\[
D_m x=0.
\]
\end{proposition}

\begin{proof}
Here, in the proof of Proposition~\ref{proposition:2.2}, by the mod $2$ cohomology ring, we mean the mod $2$ cohomology ring of $B(\mathbb{Z}/2)^2$
unless otherwise stated explicitly. 
Recall that 
\[
Q_m(x\cdot y)=(Q_m x) \cdot y+ x \cdot (Q_my),
\]
for $x, y \in H^{*}(X;\mathbb{Z}/2)$.
For $i=1$ and $2$, we have
\begin{align*}
D_m s_i&=Q_m  s_i+ B\iota^*(w_2^{2^{m-1}} ) Q_{m-1} s_i+ B\iota^{*}(w_3^{2^{m-1}}) Q_{m-2} s_i 
\\
&=s_i^{2^{m+1}} +(s_1^2+s_1s_2+s_2^2)^{2^{m-1}} s_i^{2^{m}} +(s_1^2 s_2+ s_1 s_2^2)^{2^{m-1}} s_i^{2^{m-1}}\\
&=\left( s_i^4 +(s_1^2+s_1s_2+s_2^2) s_i^{2} + (s_1^2 s_2+ s_1 s_2^2) s_i\right)^{2^{m-1}}
\\
&=0.
\end{align*}
Thus, for elements $x$, $y$ in the mod $2$ cohomology ring, we have
\[
D_m(x\cdot y)=D_m x\cdot y+ x \cdot D_m y.
\]
Therefore, since the mod $2$ cohomology ring is generated by $s_1$, $s_2$, 
the fact that $D_m s_i =0$ for $i=1,2$ implies that $D_m x=0$ for each element 
$x$ in the mod $2$ cohomology ring. 
\end{proof}



Now,  for $m\geq 2$, we describe the action of the Milnor operation $Q_m$ in terms of certain polynomials $f_{m,0}$, $f_{m.1}$ in $w_2^2$ and $w_3^2$ and Milnor operations $Q_0$, $Q_1$. Since the induced homomorphism 
\[
B\iota^{*}\colon H^{*}(BSO(3);\mathbb{Z}/2)
\to H^{*}(B(\mathbb{Z}/2)^2;\mathbb{Z}/2)
\]
is injective, by Proposition~\ref{proposition:2.2}, for each $x$ in the mod $2$ cohomology of $BSO(3)$, we have 
\[
Q_m x= w_2^{2^{m-1}} Q_{m-1}x+ w_3^{2^{m-1}} Q_{m-2}x.
\]
We may write it in the following form.
\[
\begin{pmatrix}
Q_m x\\ Q_{m-1}x
\end{pmatrix}
=
\begin{pmatrix} w_2^{2^{m-1}} & w_3^{2^{m-1}} \\ 1 & 0 \end{pmatrix}
\begin{pmatrix}
Q_{m-1}x \\ Q_{m-2}x
\end{pmatrix}.
\]
Let us define a matrix $A_m$ whose coefficients are polynomials in $w_2^2$, $w_3^2$ as follows:
\[
A_m =\begin{pmatrix} w_2^{2^{m-1}} & w_3^{2^{m-1}} \\ 1 & 0  \end{pmatrix}
\begin{pmatrix} w_2^{2^{m-2}} & w_3^{2^{m-2}} \\ 1 & 0  \end{pmatrix}
\cdots \begin{pmatrix} w_2^{4} & w_3^{4} \\ 1 & 0  \end{pmatrix}
\begin{pmatrix} w_2^{2} & w_3^{2} \\ 1 & 0  \end{pmatrix}.
\]
Furthermore, let us define polynomials $f_{m,0}$, $f_{m,1}$ by 
\[
\begin{pmatrix} f_{m,1} & f_{m,0} \end{pmatrix}= 
\begin{pmatrix}1 & 0  \end{pmatrix} A_m.
\]
Then, for $x$ in the mod $2$ cohomology of $BSO(3)$, 
we have 
\[
Q_m x= \begin{pmatrix} 1 & 0 \end{pmatrix} \begin{pmatrix} Q_{m} x\\ Q_{m-1} x \end{pmatrix}  =\begin{pmatrix} 1 & 0 \end{pmatrix} A_m \begin{pmatrix} Q_1 x\\ Q_0 x \end{pmatrix} =f_{m,1} Q_1x + f_{m,0} Q_0x.
\]




\begin{proof}[Proof of Lemma~\ref{lemma:2.1}]
We have the following congruence.
\[
A_m \equiv \begin{pmatrix} w_2^{2^{m-1}} & 0 \\ 1 & 0  \end{pmatrix}
\cdots
\begin{pmatrix} w_2^{2} & 0 \\ 1 & 0  \end{pmatrix}
\equiv \begin{pmatrix}
w_2^{2^{m}-2} & 0 \\
w_2^{2^{m-1}-2} & 0 
\end{pmatrix} \mod (w_3^2).
\]
Hence, we have $f_{m,0}\equiv 0 \mod (w_3^2)$. Therefore, there exists a polynomial $g_m$ in $w_2^2$ and $w_3^2$ such that 
\[
f_{m,0}=g_m w_3^2.
\]
Recall the fact that $Q_0Q_1 w_2=w_3^2$ and $Q_1 Q_1=0$. 
Then, we have
\begin{align*}
Q_mQ_1 w_2&=  f_{m,1}Q_1 Q_1w_2 + f_{m,0} Q_0Q_1 w_2
\\
&=f_{m,0} Q_0Q_1 w_2
\\
&=g_m  w_3^4. \qedhere 
\end{align*}
\end{proof}



\begin{example}
For $m=2, 3, 4$, elements $Q_mx$ and polynomials $g_m$ in Lemma~\ref{lemma:2.1} are as follows:
\begin{align*}
Q_2 x&= w_2^2 Q_1 x+ w_3^2 Q_0 x, & g_2&=1, \\
Q_3 x&= (w_2^6+w_3^4) Q_1x + w_2^4 w_3^2 Q_0 x, & g_3 &= w_2^4, \\
Q_4 x &= (w_2^{14}+w_2^8 w_3^4+ w_2^2 w_3^8) Q_1 x + (w_2^{12}+w_3^8)w_3^2 Q_0 x, & g_4&= w_2^{12}+w_3^8.
\end{align*}
\end{example}

\section{The nonzero odd degree element}   \label{section:3}



In this section, we prove Theorem~\ref{theorem:1.4} as Proposition~\ref{proposition:3.5}.

We begin with recalling the definition of the connected Lie group $G$ in Section~\ref{section:1} and set up notations.
Let us consider the $3$-fold product of symplectic groups $Sp(1)\subset \mathbb{H}$ consisting of 
unit quaternions.
Let
\[
\Gamma_3=\{ (\pm 1, \pm 1, \pm 1 )\}
\]
be the center of $Sp(1)^3$.
Let
$
\Gamma_2
$
be its subgroup generated by $ (-1, 1, -1), (1, -1, -1)$ 
and
\[
G=Sp(1)^3/\Gamma_2.
\]
Let $\mathbb{Z}/2=\{ (\pm 1, 1, 1) \} \subset \Gamma_3$. Then, 
$\mathbb{Z}/2$ and $\Gamma_2$ generate $\Gamma_3$. Moreover, we have
\[
Sp(1)^3/\Gamma_3=SO(3)^3.
\]
Therefore, we have the following fiber sequence:
\[
B\mathbb{Z}/2 \to BG\to BSO(3)^3.
\]

We denote by 
\[
\{ 
{E}_r^{p,q}, {d}_r\colon {E}_r^{p,q}\to {E}_r^{p+r,q-r+1}
\}
\]
 the Leray-Serre spectral sequence associated with this fiber sequence.
 Let us denote its $E_r$-term by 
 \[
E_r= \bigoplus_{p,q} E_r^{p,q}.
 \]
We compute the mod $2$ cohomology of $BG$ using the above Leray-Serre spectral sequence.
Although it is easy, we quickly review it. See \cite{kameko-preprint} for more detail.



We describe the $E_2$-term and compute the first non-trivial differential $d_2$.
Let 
\[
B\pi_i\colon BSO(3)^3\to BSO(3)
\]
be the map induced by the projection to the $i^{\mathrm{th}}$ factor for $i=1, 2, 3$.
We denote by $w_i'$, $w_i''$, $w_i'''$ the cohomology classes $B\pi_1^{*}(w_i)$, $B\pi_2^{*}(w_i)$, $B\pi_3^{*}(w_{i})$, respectively. Let $u_1$ be the generator of the mod $2$ cohomology $H^{1}(B\mathbb{Z}/2;\mathbb{Z}/2)\cong\mathbb{Z}/2$ of the fibre $B\mathbb{Z}/2$.
The $E_2$-term is given by
\begin{align*}
E_2&=\mathbb{Z}/2[ w_2', w_3', w_2'', w_3'', w_2''', w_3''']\otimes \mathbb{Z}/2[u_1]
\end{align*}
To compute the differential ${d}_2$, we consider the Leray-Serre spectral sequence 
\[
\{ 
\bar{E}_r^{p,q}, \bar{d}_r\colon \bar{E}_r^{p,q}\to \bar{E}_r^{p+r,q-r+1}
\}
\]
associated with the fiber sequence
\[
B\mathbb{Z}/2 \to BSp(1) \to BSO(3).
\]
Recall that its $E_2$-term is given as follows:
\[
\bar{E}_2=\mathbb{Z}/2[ w_2, w_3] \otimes \mathbb{Z}/2[u_1].
\]
and its first nontrivial differential $\bar{d}_2$ is given by
\[
\bar{d}_2(u_1)=w_2.
\]
Let 
\[
B\iota_i\colon BSp(1) \to BG
\]
be the map induced by the inclusion map $\iota_i$ of $Sp(1)$ for $i=1,2,3$ such  that 
\[
\iota_1(g)=(g, 1, 1), \quad \iota_2(g)=(1, g, 1), \quad \iota_3(g)=(1, 1, g).
\]
Then we have the following commutative diagram,
\[
\begin{diagram}
\node{B\mathbb{Z}/2} \arrow{e} \arrow{s,l}{=} \node{BSp(1)} \arrow{e} \arrow{s,r}{B\iota_i} \node{BSO(3)} \arrow{s,r}{B\iota_i} \\
\node{B\mathbb{Z}/2} \arrow{e}  \node{BG} \arrow{e}  \node{BSO(3)^3.}
\end{diagram}
\]
Furthermore, we have 
\[
\begin{array}{lll}
B\iota_1^*(w_2')= w_2, & B\iota_1^*(w_2'')= 0, & B\iota_1^*(w_2''')= 0, \\ 
B\iota_2^*(w_2')= 0, & B\iota_2^*(w_2'')= w_2, & B\iota_2^*(w_2''')= 0, \\ 
B\iota_3^*(w_2')= 0, & B\iota_3^*(w_2'')= 0, & B\iota_3^*(w_2''')= w_2.
\end{array}
\]
Now, we are ready to compute the differential $d_2$.
Suppose that the first nontrivial differential $d_2$ is given as follows:
\begin{align*}
d_2(u_1)&=\alpha_1 w_2'+\alpha_2 w_2''+ \alpha_3w_2''',
\end{align*}
where $\alpha_1, \alpha_2, \alpha_3$ are in $\mathbb{Z}/2$.
Since 
\[
B\iota_i^*(d_2(u))=\alpha_i w_2, 
\]
and 
\[
\bar{d}_2(u_1)=w_2, 
\]
we obtain 
\[
\alpha_i=1
\]
for $i=1,2,3$. 
Thus, 
the first nontrivial differential $d_2\colon E_2^{0,1}\to E_2^{2,0}$ is given as follows:
\begin{align*}
d_2(u_1)&=w_2'+w_2''+ w_2'''.
\end{align*}



Let us recall the relation between the transgression and Steenrod squares. 
For $r\geq 2$, the transgression 
\[
d_r\colon E_r^{0,r-1}\to E_r^{r,0}
\]
commutes with Steenrod squares $\mathrm{Sq}^i$.  In other words, if $d_r(x)=y$ then
we may have an element $\mathrm{Sq}^i x \in E_s^{0,r-1+i}$ for $r\leq s $, 
an element  $\mathrm{Sq}^i y\in E_{r+i}^{r+i, 0}$ and there hold that
$d_{s}(\mathrm{Sq}^i x)=0$ for $r\leq s<r+i$ and that $d_{r+i}(\mathrm{Sq}^i x)=\mathrm{Sq}^i y$. 

Starting with the above $E_2$ and $d_2$, since $u_1^{2}=\mathrm{Sq}^1 u_1$, $u_1^{4}=\mathrm{Sq}^2 u_1^2$, and $u_1^{8}=\mathrm{Sq}^4 u_1^4$,
we have the following $E_r$-tems and differentials up to $r\leq 9$.
\begin{align*}
E_3&=\mathbb{Z}/2[w_2', w_3',  w_2'', w_3'', w_3''']\otimes \mathbb{Z}/2[ u_1^2], \\
d_3(u_1^2)&=\mathrm{Sq}^1(w_2' +w_2''+ w_2''')
\\
&=w_3' +w_3''+ w_3''', \\
E_4&=\mathbb{Z}/2[w_2', w_3',  w_2'', w_3'']\otimes \mathbb{Z}/2[ u_1^4], \\
d_4(u_1^4)&=0, \\
E_5&=E_4, \\
d_5(u_1^4)&=\mathrm{Sq}^2 (w_3' +w_3''+ w_3''') \\
&=w_2'w_3'+ w_2'' w_3'' + w_2''' w_3'''
\\
& =w_2' w_3''+ w_2'' w_3', 
\\
E_6&=\mathbb{Z}/2[w_2', w_3', w_2'', w_3'']/(w_2' w_3''+ w_2'' w_3') \otimes \mathbb{Z}/2[ u_1^8], \\
d_6(u_1^8)&=0, \\
d_7(u_1^8)&=0, \\
d_8(u_1^8)&=0, \\
E_9&=E_6.
\end{align*}
To compute higher terms and differentials, 
let us consider the ring homomorphism
\[
\phi\colon \mathbb{Z}/2[ w_2', w_3',  w_2'', w_3'']\to \mathbb{Z}/2[w_2', w_2'', t_1]
\]
defined by
\begin{align*}
\phi(w_2')&=w_2', \\
\phi(w_3')&=t_1 w_2', \\
\phi(w_2'')&=w_2'', \\
\phi(w_3'')&=t_1 w_2''.
\end{align*}
We assign weight $0$, $1$, $0$, $1$ to $w_2'$, $w_3'$, $w_2''$, $w_3''$, respectively. 
We also assign weight $1$ to $t_1$. 
Then, the ring homomorphism $\phi$ is weight-preserving.

Let 
\[
M=\mathbb{Z}/2[ w_2', w_3', w_2'' , w_3'']/( w_2'w_3''+w_2''w_3')
.
\]
It is the bottom line of the $E_9$-term of the spectral sequence such that 
\[
M=\bigoplus_{p} E_9^{p,0}.
\]
The ring homomorphism $\phi$ induces the weight-preserving ring homomorphism
\[
\bar{\phi}\colon M \to \mathbb{Z}/2[w_2', w_2'', t_1].
\]
It is clear that the ring homomorphism $\bar{\phi}$ is injective. 
Thus, $M$ is isomorphic to the subring  $\bar{\phi}(M)$ of $\mathbb{Z}/2[ w_2', w_2'', t_1]$. Therefore, both $M$ and $\bar{\phi}(M)$ are integral domains. 

The next nontrivial differential is $d_9$. It is given by 
\begin{align*}
d_9(u_1^8)&=\mathrm{Sq}^4 (w_2' w_3''+ w_2'' w_3')
\\
&=w_2'^2 w_2'' w_3''+w_3' w_3''^2 +w_2''^2 w_2' w_3'+w_3'' w_3''^2
\\
&=w_2'w_2''(w_2'w_3''+w_2''w_3')+w_3' w_3''^2 + w_3'' w_3'^2
\\
&=w_3' w_3''^2 + w_3'' w_3'^2.
\end{align*}
Since 
\[
\bar{\phi}(w_3' w_3''^2 + w_3'' w_3'^2)=t_1^3 w_2'w_2''(w_2'+w_2'')
\]
is nonzero 
in $\mathbb{Z}/2[ w_2', w_2'', t_1]$, multiplication by $w_3' w_3''^2 + w_3'' w_3'^2$ is injective on $M$. Therefore, we have 
\begin{align*}
E_{10}=& \mathbb{Z}/2[ w_2', w_3', w_2'', w_3'']/(w_2'w_3''+ w_2'' w_3', w_3'w_3''^2+ w_3'' w_3'^2) \otimes \mathbb{Z}/2[u_1^{16}]. \end{align*}
We would like to point out that $w_2'w_3''+ w_2'' w_3', w_3'w_3''^2+ w_3'' w_3'^2$ is a regular sequence in the polynomial ring $\mathbb{Z}/2[w_2', w_2'', w_3', w_3'']$.

Finally, using the commutativity between the transgression and Steenrod squares again, we have
\begin{align*}
d_r(u_1^{16})&=0 \quad \mbox{for $10\leq r \leq 16$ and}
\\
d_{17}(u_1^{16})&=\mathrm{Sq}^{8} (w_3'w_3''^2 + w_3'^2w_3'')
\\
&=w_2'w_3' w_3''^4+ w_2'' w_3'' w_3'^4 
\\
&=(w_2' w_3''+w_2'' w_3') w_3' w_3''^3+ w_2'' w_3'(w_3'+w_3'') (w_3' w_3''^2+w_3'^2 w_3'')
\\
&=0.
\end{align*}
Hence, we have $d_r=0$ for $r\geq 10$ and $E_\infty=E_{10}$.


To describe the $E_\infty$-term, let 
\[
N=\mathbb{Z}/2[ w_2', w_3', w_2'', w_3'']/(w_2' w_3''+w_2''w_3', w_3'w_3''^2 + w_3'' w_3'^2).
\]
It is the bottom line of the $E_\infty$-term of the spectral sequence such that 
\[
N=\bigoplus_p E_\infty^{p,0}.
\]
It is also the subring
of the mod $2$ cohomology ring of $BG$ 
generated by $w_2', w_3', w_2'', w_3''$.
What we need is the fact that the induced homomorphism 
\[
N\to H^{*}(BG;\mathbb{Z}/2)
\]
is injective, and $N$ is closed under the action of Milnor operations $Q_m$ for $m\geq 0$.


For a graded set $\{ x_1, x_2, \dots \}$, we denote by $\mathbb{Z}/2\{ x_1, x_2, \dots \}$ the graded $\mathbb{Z}/2$-module spanned by
$\{ x_1, x_2, \dots \}$.
Recall that we defined weight of $w_2', w_3', w_2'', w_3'', t_1$ as $0,1,0,1,1$, respectively. 
We have direct sum decompositions of $M$ and $N$ with respect to weight. 
Namely, $M_k$, $N_k$ are graded submodules of $M$, $N$ spanned by monomials of weight $k$, respectively. 

We will define the element $x_{13}$ as an element in $N_1$. We also need the following Proposition~\ref{proposition:3.1}  on the basis for $N_1$ to show that $x_{13}$ is nonzero.


\begin{proposition}\label{proposition:3.1} For $N_0, N_1, N_2$, we have
\begin{align*}
N_0&=\mathbb{Z}/2\{ w_2'^mw_2''^n\;|\; m, n \geq 0\}, \\
N_1&=\mathbb{Z}/2\{ w_2'^m w_3', w_2'^m w_2''^n w_3''\; |\; m, n \geq 0\}, \\
N_2&=\mathbb{Z}/2\{ w_2'^m w_3'^2, w_2'^m w_2''^n w_3' w_3'', w_2''^n w_3''^2 \;|\; m, n \geq 0\}.
\end{align*}
\end{proposition}

\begin{proof}
The weight of monomials in 
\[
\bar{\phi}(w_3'w_3''^2+w_3'^2 w_3'')=t_1^3 w_2' w_2''^2+t_1^3 w_2'^2 w_2''
\]
 is $3$. Therefore, the ideal of $M$ generated by 
\[
 w_3'w_3''^2+w_3'^2 w_3''
\]
is spanned by monomials of weight greater than or equal to $3$. Hence, we have. $N_i=M_i$ for $i=0,1,2$. 
It is clear that
\begin{align*}
\bar{\phi}(M_0)&=\mathbb{Z}/2\{ w_2'^m w_2''^n\}, \\
\bar{\phi}(M_1)&=\mathbb{Z}/2\{ t_1 w_2'^m w_2''^n\;|\; m+n \geq 1 \}, \\
\bar{\phi}(M_2)&=\mathbb{Z}/2\{ t_1^2 w_2'^m w_2''^n\;|\; m+n \geq 2 \}
\end{align*}
and that
\begin{align*}
\bar{\phi}(\mathbb{Z}/2\{ w_2'^mw_2''^n\} )&=\mathbb{Z}/2\{ w_2'^m w_2''^n\}, \\
\bar{\phi}(\mathbb{Z}/2\{ w_2'^m w_3', w_2'^m w_2''^n w_3''\} )&=\mathbb{Z}/2\{ t_1 w_2'^m w_2''^n\;|\; m+n \geq 1\}, \\
\bar{\phi}(\mathbb{Z}/2\{ w_2'^m w_3'^2, w_2'^m w_2''^n w_3' w_3'', w_2''^n w_3''^2 \})&=\mathbb{Z}/2\{ t_1^2 w_2'^m w_2''^n\;|\; m+n \geq 2\},
\end{align*}
where $m, n$ range over the set of nonnegative integers.
Since the ring homomorphism $\bar{\phi}$ is injective, 
we obtain the desired results.
\end{proof}


We need the following lemma on $N_k$ ($k\geq 3$) to show that $Q_m x_{13}=0$ for $m\geq 2$.

\begin{proposition}\label{proposition:3.2}
Suppose that $k\geq 3$. For $1\leq i \leq k-1$, $m\geq 0$, $n\geq 0$, we have
\[
w_2'^m w_2''^n w_3'^i w_3''^{k-i} = w_2''^{m+n}w_3' w_3''^{k-1}
\]
in $N_k$.
\end{proposition}

\begin{proof}
For $i\geq 2$, we have
\begin{align*}
w_3'^i w_3''^{k-i}&=w_3'^2 w_3'' \cdot w_3'^{i-2} w_3''^{k-i-1} \\
&=w_3' w_3''^2 \cdot w_3'^{i-2} w_3''^{k-i-1}  \quad  (\because w_3'^2 w_3''=w_3'w_3''^2)\\
&=w_3'^i w_3''^{k-i}.
\end{align*}
Iterating this process, we have $w_3'^i w_3''^{k-i}=w_3'w_3''^{k-1}$.
For $m\geq 1$, we have
\begin{align*}
w_2'^m w_2''^n w_3' w_3''^{k-1} &= w_2' w_3'' \cdot w_2'^{m-1} w_2''^n w_3' w_3''^{k-2} \\
&=w_3' w_2'' \cdot w_2'^{m-1} w_2''^n w_3' w_3''^{k-2}   \quad (\because w_2' w_3''= w_3' w_2'') \\
&=w_2'^{m-1} w_2''^{n+1} w_3'^2 w_3''^{k-2} \\
&=w_2'^{m-1} w_2''^{n+1} w_3' w_3''^{k-1}  \quad (\because w_3'^2 w_3^{k-2}=w_3' w_3''^{k-1}).
\end{align*}
Hence, we have the desired result
$
w_2'^m w_2''^n w_3'^i w_3''^{k-i} = w_2''^{m+n}w_3' w_3''^{k-1}
$.
\end{proof}


\begin{remark}\label{remark:3.3}
With Proposition~\ref{proposition:3.2}, it is easy to find a basis for $N_k$. 
And we have the following.
\[ N_k =\mathbb{Z}/2\{ w_2'^m w_3'^k, w_2''^n w_3' w_3''^{k-1}, w_2''^n w_3''^k \;|\; m, n \geq 0 \}.
\]
\end{remark}


\begin{remark}\label{remark:3.4}
It is  easy to compute the Poincar\'{e} series 
\[
\dfrac{(1-t^5)(1-t^9)}{(1-t^2)^2(1-t^{3})^2}                                                                                                                                                                                                                                                                                                                                                                                                                                                                                                                                                                                                                                                                                          
\]
of $N$ since $w_2'w_3''+ w_2'' w_3'$, $w_3' w_3''^2 + w_3'' w_3'^2$ is a regular sequence.
To prove the linear independence of elements in Propositions~\ref{proposition:3.1} and \ref{proposition:3.2}, 
one may compute the Poincar\'{e} series of each $N_k$ and add them up to obtain the Poincar\'{e} series of $N$ above. \end{remark}



\begin{proposition}\label{proposition:3.5}
Let us define an element $x_{13}$ of degree $13$ in the mod $2$ cohomology of $BG$ by
\[
x_{13}:= B\pi_1^{*}(Q_1 w_2) w_2''^2 (w_2'^2 +w_2''^2).
\]
Then, $x_{13}$ is nonzero and for $m\geq 1$, we have
\[
Q_m x_{13}=0.
\]
\end{proposition}

\begin{proof}
First, we verify that $x_{13}$ is nonzero. 
Since $B\pi_1^{*}(Q_1 w_2)=w_2' w_3'$, we have
\begin{align*}
x_{13}&=w_2'^3 w_2''^2 w_3'+ w_2' w_2''^4 w_3'\\
&=w_2'^4 w_2'' w_3''+ w_2'^2 w_2''^3 w_3''
\\
&\not=0
\end{align*}
in $N_1$ by Proposition~\ref{proposition:3.1}.
Next, we compute $Q_mx_{13}$. Since $Q_m$ acts trivially on $w_2''^2(w_2'^2+w_2''^2)$, 
\begin{align*}
Q_m x_{13}&= B\pi_1^*(Q_m Q_1 w_2 ) w_2''^2 (w_2'^2 + w_2''^2).
\end{align*}
For $m=1$, since $Q_1Q_1=0$, we have $Q_1x_{13}=0$.
For $m\geq 2$, by Lemma~\ref{lemma:2.1}, we have
\begin{align*}
 B\pi_1^*(Q_m Q_1 w_2 ) w_2''^2 (w_2'^2 + w_2''^2)&=B\pi_1^{*}(g_m w_3^4)w_2''^2 (w_2'^2 + w_2''^2)
 \\
 &=B\pi_1^{*}(g_m) w_3'^4 w_2''^2 (w_2'^2 + w_2''^2).
\end{align*}
By Proposition~\ref{proposition:3.2}, we obtain
\begin{align*}
w_3'^4 w_2''^2 w_2'^2 &= w_2''^4 w_3' w_3''^3, \\
w_3'^4 w_2''^2 w_2''^2 &= w_2''^4 w_3' w_3''^3, 
\end{align*}
hence, we have 
\[
w_3'^4 w_2''^2 (w_2'^2 + w_2''^2)=0.
\]
Therefore, we obtain
$
Q_mx_{13}=0
$.
\end{proof}


\end{document}